\documentclass[12pt]{article}

\usepackage[latin1]{inputenc}
\usepackage[T1]{fontenc}
\usepackage[english]{babel}
\usepackage[top=3cm, bottom=3cm, left=3cm, right =3cm]{geometry}
\usepackage{amsthm}  
\usepackage{amssymb}
\usepackage{eufrak}
\usepackage{MnSymbol}
\usepackage{xcolor}

\author{Benjamin \textsc{Druart}\thanks{Université de Grenoble I, Département de Mathématiques, Institut Fourier, UMR 5582 du CNRS, 38402 Saint-Martin d'Hères Cedex, France. email : Benjamin.Druart@ujf-grenoble.fr}
\thanks{The research leading to these results has recieved funding from the European Research Council under the European Community's Seventh Framework Programme FP7/2007-2013 Grant Agreement no. 278722.}}
\title{Cartan Subgroups and Generosity in $SL_2(\Qp)$}

\newtheorem{defi}{Definition}
\newtheorem{fact}{Fact}
\newtheorem{prop}[defi]{Proposition}
\newtheorem{cor}[defi]{Corollary}
\newtheorem*{rmk}{Remark}
\newtheorem{lem}{Lemma}
\newtheorem{thm}[defi]{Theorem}

\newcommand{\Qp}{\mathbb{Q}_p}
\newcommand{\Zp}{\mathbb{Z}_p}
\newcommand{\Z}{\mathbb{Z}}

\begin{document}

\maketitle

\begin{abstract}
We show that there exist a finite number of Cartan subgroups up to conjugacy in $SL_2(\Qp)$ and we describe all of them. We show that the Cartan subgroup consisting of all diagonal matrices is generous and it is the only one up to conjugacy.
\end{abstract}

\bigskip

A subset $X$ of a group $G$ is \emph{left-generic} if $G$ can be covered by finitely many left-translates of $X$. 
We define similarly right-genericity. 
If $X$ is $G$-invariant, then left-genericity is equivalent to  right-genericity.
This important notion in model theory was particulary developped by B. Poizat for groups in stable theories \cite{Poizat_stable}. For a group of finite Morley-rank and $X$ a definable subset, genericity is the same as being of maximal dimension \cite[lemme 2.5]{Poizat_stable}. 
The term generous was introduced in \cite{Jal_never} to show some conjuguation theorem. A definable subset $X$ of a group $G$ is \emph{generous} in $G$ if the union of its $G$-conjugates, $X^G=\{x^g\mid (x,g)\in X\times G\}$, is generic in $G$.  

In an arbitrary group $G$, we define a \emph{Cartan subgroup} $H$ as a maximal nilpotent subgroup such that every finite index normal subgroup $X \unlhd H$ is of finite index in its normalizer $N_G(X)$. First we can remark that they are infinite because $N_G(1)=G$ and if $H$ is finite then $\{1\}$ is of finite index in $H$.
In connected reductive algebraic groups over an algebraically closed fields, the maximal torus is typically an example of a Cartan subgroup. Moreover it is the only one up to conjugation and it is generous.
It has been remarked in \cite{cartan} that, in the group $SL_2(\mathbb{R})$, also the Cartan subgroup  consisting of diagonal matrices is   generous. But it has also been remarked that in the case of $SL_2(\mathbb{R})$, there exists another Cartan subgroup, namely $SO_2(\mathbb{R})$, which is not generous.

\bigskip
We will discuss here some apparently new remarks of the same kind in $SL_2(\Qp)$. First we describe all Cartan subgroups of $SL_2(\Qp)$. After we show that the Cartan subgroup consisting of diagonal matrices is generous and it is the only one up to conjugacy.

I would like to thank E. Jaligot, my supervisor for his help, E. Baro to explain me the case of $SL_2(\mathbb{R})$, and T. Altinel to help me to correct a mistake in a previous version of this paper.

%CARTAN SUBGROUPS
\subsection*{Description of Cartan subgroups up to conjugacy}
We note $v_p : \Qp\longrightarrow \Z\bigcup\{+\infty\}$ the p-adic-valuation, and $ac : \Qp^{\times} \longrightarrow \mathbb{F}_p$ the angular component defined by $ac(x)=res(p^{-v_p(x)}x)$ where $res: \Qp\longrightarrow \mathbb{F}_p$ is the residue map. 

With these notations, if $p\neq 2$, an element $x\in \Qp^{\times}$ is a square if and only if $v_p(x)$ is even and $ac(x)$ is a square in $\mathbb{F}_p$. For $p=2$, an element $x\in\mathbb{Q}_2$ can be written $x=2^nu$ with $n\in \Z$ and $u\in \mathbb{Z}_2^{\times}$, then $x$ is a square if $n$ is even and $u\equiv 1 (\mbox{mod } 8)$ \cite{JPS}.

\begin{fact}[\cite{JPS}]\label{Serre}
If $p\neq2$, the group $\Qp^{\times} / (\Qp^{\times})^2$ is isomorphic to $\Z/2\Z \times \Z/2\Z$, it has for representatives $\{1, u, p, up\}$, where $u\in \Zp^{\times}$ is such that $ac(u)$ is not a square in $\mathbb{F}_p$

The group $\mathbb{Q}_2^{\times}/(\mathbb{Q}_2^{\times})^2$ is isomorphic to $\Z/2\Z \times \Z/2\Z \times \Z/2\Z$, it has for representatives $\{\pm1,\pm 2, \pm 5, \pm10\}$. 
\end{fact}

For any prime $p$, and any $\delta$ in $\Qp^{\times}\backslash (\Qp^{\times})^2$, we put : 
\begin{align}
Q_1 &=\left\{\left(\begin{array}{cc}a & 0 \\0 & a^{-1}\end{array}\right)\in SL_2(\Qp)\mid a\in \Qp^{\times}\right\} \nonumber\\ 
Q_{\delta}&=\left\{\left(\begin{array}{cc}a & b \\b\delta & a\end{array}\right)\in SL_2(\Qp) \mid a,b\in \Qp \mbox{ and } a^2-b^2 \delta=1\right\} \nonumber
\end{align}

\begin{lem}\label{centre}
\begin{align} 
 \forall x\in Q_1\backslash \{I,-I\} \quad C_{SL_2(\Qp)}(x)=Q_1\nonumber
\\ \forall x\in Q_{\delta}\backslash \{I,-I\} \quad C_{SL_2(\Qp)}(x)=Q_{\delta} \nonumber
\end{align}
\end{lem}

The checking of these equalities is left to the reader.

\begin{prop}
The groups $Q_1$ and $Q_{\delta}$ are Cartan subgroups of $SL_2(\Qp)$ 
\end{prop}

\begin{proof}
One checks easily that $Q_1$ is abelian and the normalizer of $Q_1$ is :
$$N_{SL_2(\Qp)}(Q_1)=Q_1\cdot<\omega> \quad\mbox{ where } \quad\omega =\left(\begin{array}{cc}0 & 1 \\-1 & 0\end{array}\right)$$
For $X$ a subgroup of $Q_1$, if $g\in N_{SL_2(\Qp)}(X)$ and $x\in X$, then, using lemma \ref{centre}
$$Q_1 = C_{SL_2(\Qp)}(x)=C_{SL_2(\Qp)}(x^g)=C_{SL_2(\Qp)}(x)^g=Q_1^g$$ 
It follows that $N_{SL_2(\Qp)}(X)=N_{SL_2(\Qp)}(Q_1)=Q_1\cdot <\omega>$ and if $X$ of finite index $k$ in $Q_1$, then $X$ is of index $2k$ in $N_{SL_2(\Qp)}(X)$. 
We can see that for $t$ in $Q_1$, $t^{\omega}=\omega^{-1}t\omega =t^{-1}$ and thus $[\omega , t]=t^2$.

If we note $\Gamma_i$ the descending central series of $N_{SL_2(\Qp)}(Q_1)$, we have $\Gamma_0=N_{SL_2(\Qp)}(Q_1)$, $\Gamma_1=[N_{SL_2(\Qp)}(Q_1),N_{SL_2(\Qp)}(Q_1)]=Q_1^2$ and $\Gamma_i=[N_{SL_2(\Qp)}(Q_1),\Gamma_{i-1}]=Q_1^{2^i}$. 
Observing that $Q_1\cong \Qp^{\times}$, we can conclude that the serie of $\Gamma_i$ is infinite because $Q_1$ is infinite and $Q_1^{2^i}$ is of finite index in $Q_1^{2^{i-1}}$. Thus $N_{SL_2(\Qp)}(Q_1)$ is not nilpotent.
By the normalizer condition for nilpotent groups, if $Q_1$ is properly contained in a nilpotent group $K$, then $Q_1 < N_K(Q_1) \leq K$, here $N_K(Q_1)=Q_1\cdot<\omega>$ which is not nilpotent, a contradiction. It finishes the proof that $Q_1$ is a Cartan subgroup.

For $\delta \in \Qp^{\times}\backslash (\Qp^{\times})^2$, we check similarly  that the group $Q_{\delta}$ is abelian. Since for all subgroups $X$ of $Q_{\delta}$, $C_{SL_2(\Qp)}(X)=Q_{\delta}$, it follows that $N_{SL_2(\Qp)}(X)=N_{SL_2(\Qp)}(Q_{\delta})=Q_{\delta}$, and if $X$ is of finite index in $Q_{\delta}$ then $X$ is of finite index in its normalizer.  By the normalizer condition for nilpotent groups, $Q_{\delta}$ is nilpotent maximal. 
\end{proof}

\begin{prop}\label {tr}
\begin{enumerate}
\item $Q_1^{SL_2(\Qp)} = \left\{ A\in SL_2(\Qp)\mid tr(A)^2-4 \in (\Qp^{\times})^2 \right\} \bigcup \left\{I,-I\right\}$

\item For any $\delta \in \Qp^{\times}\backslash(\Qp^{\times})^2$, there exist $\mu_1, ... \mu_n\in GL_2(\Qp)$ such that $$\bigcup _{i=1}^{n}Q_{\delta}^{\mu_i\cdot SL_2(\Qp)} = \left\{ A\in SL_2(\Qp)\mid tr(A)^2-4 \in \delta\cdot(\Qp^{\times})^2 \right\} \bigcup \left\{I,-I\right\}$$
\end{enumerate}
\end{prop}

We put : 
$$ U=\left\{\left(\begin{array}{cc}1 & u \\0 & 1\end{array}\right)\mid u\in\Qp\right\}\bigcup\left\{\left(\begin{array}{cc}-1 & u \\0 & -1\end{array}\right)\mid u\in \Qp\right\} \mbox{ and } U^+=\left\{\left(\begin{array}{cc}1 & u \\0 & 1\end{array}\right)\mid u\in\Qp\right\}$$

If $A\in SL_2(\Qp)$ satisfies $tr(A)^2-4=0$, then either $tr(A)=2$ or $tr(A)=-2$, and $A$ is a conjugate of an element of $U$. In this case, $A$ is said \textit{unipotent}.  It follows, from Proposition \ref{tr} : 

\begin{cor}
We have the following partition : 
$$SL_2(\Qp)\backslash \{I,-I\} = (U\backslash \{I,-I\}) ^{SL_2(\Qp)}\sqcup (Q_1\backslash \{I,-I\})^{SL_2(\Qp)} \sqcup \bigsqcup_{\delta \in \Qp^{\times}/(\Qp^{\times})^2}\bigcup_{i=1}^{n}(Q_{\delta}^{\mu_i} \backslash \{I,-I\})^{SL_2(\Qp)} \label{union}$$
\end{cor}

\begin{rmk}
If $\delta$ and $\delta'$ in $\Qp^{\times}$ are in the same coset of $(\Qp^{\times})^2$, then, by Proposition \ref{tr}, if $x'\in Q_{\delta'}^{\mu'}$ with $\mu'\in GL_2(\Qp)$, then there exists $x\in Q_{\delta}$, $\mu\in GL_2(\Qp)$ and $g\in SL_2(\Qp)$, such that $x'=x^{\mu\cdot g}$, thus, by lemma \ref{centre}, $Q_{\delta'}=C_{SL_2(\Qp)}(x')=C_{SL_2(\Qp)}(x)^{\mu\cdot g}=Q_{\delta}^{\mu \cdot g}$.
Therefore the Corollary \ref{union} makes sense.
\end{rmk}

\begin{proof}[Proof of Proposition \ref{tr}]
$\bullet$ If $A\in Q_1^{SL_2(\Qp)}$, then there exists $P\in SL_2(\Qp)$ such that 
$$A=P \left(\begin{array}{cc}a & 0 \\0 & a^{-1}\end{array}\right)P^{-1}$$
with $a\in \Qp^{\times}$. We have $tr(A)=a+a^{-1}$, so $tr(A)^2-4=(a+a^{-1})^2-4=(a-a^{-1})^2$ and $tr(A)^2-4\in (\Qp^{\times})^2$.

Conversely , let $A$ be in $SL_2(\Qp)$ with $tr(A)^2-4$ a square. The caracteristic polynomial is $\chi_A(X)=X^2-tr(A)X+1$ and its discriminant is $\Delta=tr(A)^2-4\in (\Qp^{\times})^2$, so $\chi_A$ has two distinct roots in $\Qp$ and $A$ is diagonalizable in $GL_2(\Qp)$. There is $P\in GL_2(\Qp)$, and  $D\in SL_2(\Qp)$ diagonal such that $A=PDP^{-1}$. If
$$P=\left(\begin{array}{cc}\alpha & \beta \\\gamma & \delta \end{array}\right)$$ 
we put 
$$\tilde{P}=\left(\begin{array}{cc}\frac{\alpha}{det(P)} & \beta \\\frac{\gamma}{det(P)} & \delta\end{array}\right)$$ 
and we have $\tilde{P}\in SL_2(\Qp)$ and $A=\tilde{P}D\tilde{P}^{-1} \in Q_1^{SL_2(\Qp)}$.

$\bullet$
If $A$ is in $Q_{\delta}^{\mu\cdot SL_2(\Qp)}\backslash \{I,-I\}$ with $\mu \in GL_2(\Qp)$, then $tr(A)=2a$ and there exists $b\neq 0$ such that $a^2-b^2\delta=1$. 
So $tr(A)^2-4 = 4a^2-4=4(b^2\delta +1)-4=(2b)^2\delta \in \delta\cdot(\Qp^{\times})^2$

Conversely we proceed as in the real case and the root $i\in\mathbb{C}$. The discriminant of $\chi_A$, $\Delta=tr(A)^2-4$ is a square in $\Qp(\sqrt{\delta})$, and the caracteristic polynomial $\chi _A$ has two roots in $\Qp(\sqrt{\delta})$ : $\lambda_1 = \alpha + \beta \sqrt{\delta}$ and $\lambda_2=\alpha -\beta \sqrt{\delta}$ (with $\alpha, \beta \in \Qp$). For the two eigen values $\lambda_1$ and $\lambda_2$, A has eigen vectors : 
$$v_1=\left(\begin{array}{c}x+y\sqrt{\delta} \\x'+y'\sqrt{\delta}\end{array}\right) \quad \mbox{ and } \quad v_2=\left(\begin{array}{c}x-y\sqrt{\delta} \\x'-y'\sqrt{\delta}\end{array}\right)$$
 
In the basis $\left\{\left(x, x'\right),\left(y, y'\right)\right\}$, the matrix A can be written : 
$$\left(\begin{array}{cc}a & b \\b\delta & a\end{array}\right)$$ 
with $a, b\in \Qp$. We can conclude that there exists $P\in GL_2(\Qp)$ such that :
$$A=P\left(\begin{array}{cc}a & b \\b\delta & a\end{array}\right)P^{-1}$$
We proved that $Q_{\delta}^{GL_2(\Qp)}= \left\{ A\in SL_2(\Qp)\mid tr(A)^2-4 \in \delta\cdot(\Qp^{\times})^2 \right\} \bigcup \left\{I,-I\right\}$.

Let us now study the conjugation in $GL_2(\Qp)$ and in $SL_2(\Qp)$. For the demonstration, we note : $S=SL_2(\Qp)$, $G=GL_2(\Qp)$ and $Ext(S)=\{f\in Aut(S) \mid f(M)=M^P\mbox{ for } M\in S, P\in G\}$, $Int(S)=\{f\in Aut(S) \mid f(M)=M^P \mbox{ for }M\in S, P\in S\}$.
Let $P,P'\in G$ and $M\in S$ then : 
$$M^P=M^{P'} \Leftrightarrow P^{-1}MP=P'^{-1}MP' \Leftrightarrow P'P^{-1}M=MP'P^{-1} \Leftrightarrow PP'\in C_G(M)$$
So $P$ and $P'$ define the same automorphism if and only if  $P'P^{-1}\in C_G(S)=Z(G)=\Qp\cdot I_2$, then $Ext(S)\cong GL_2(\Qp)/Z(G) \cong PGL_2(\Qp)$, and similarly $Int(S)\cong SL_2(\Qp)/Z(S) \cong PSL_2(\Qp)$. 
It is known that $PGL_2(\Qp)/PSL_2(\Qp) \cong \Qp^{\times}/(\Qp^{\times})^2$.
Finally $Int(S)$ is a normal subgroup of finite index in $Ext(S)$, and there exist $\mu_1, ..., \mu_n\in GL_2(\Qp)$ such that :
$$Q_{\delta}^{GL_2(\Qp)}= Q_{\delta}^{\mu_1\cdot SL_2(\Qp)} \cup ... \cup Q_{\delta}^{\mu_n\cdot SL_2(\Qp)}$$
\end{proof}

\begin{thm}
The subgroups $Q_1$, $Q_{\delta}$ (for $\delta \in \Qp^{\times}\backslash (\Qp^{\times})^2$) and the externally conjugate $Q_{\delta}^{\mu_i}$ (for $\mu_1, ..., \mu_n\in GL_2(\Qp)$) are the only Cartan subgroups up to conjugacy of $SL_2(\Qp)$ 	
\end{thm}

\begin{proof}
It is clear that the image of a Cartan subgroup by an automorphism is also a Cartan subgroup. 
For the demonstration we note $S=SL_2(\Qp)$ and $B$ the following subgroup of $SL_2(\Qp)$ :
$$B=\left\{ \left( \begin{array}{cc}t & u \\0 & t^{-1}\end{array}\right)  \mid t\in \Qp^{\times}, u\in \Qp\right\}$$
With these notations, we can easily check for $g\in U\backslash\{I,-I\}$ that $C_S(g)=U$ and $N_S(U)=B$. Moreover it is known that every $q\in B$ can be written as $q=tu$ where $t\in Q_1$ and $u\in U$.

Consider $K$ a Cartan subgroup of $SL_2(\Qp)$. We will show that $K$ is a conjugate of $Q_1$ or of one of the $Q_{\delta}^{\mu}$ (for $\delta \in \Qp^{\times}\backslash(\Qp^{\times})^2$ and $\mu\in GL_2(\Qp)$). First we prove $K$ cannot contain a unipotent element other than $I$ or $-I$. Since a conjugate of a Cartan subgroup is still a Cartan subgroup, it suffices to show that $K\cap U=\{I,-I\}$.

In order to find a contradiction, let $u\in K$ be a element of $U$ different from $I$ or $-I$, $u$ is in $K\cap B$. If $\alpha \in N_S(K\cap B)$, then we have that $u^{\alpha}\in K\cap B$, and since $tr(u^{\alpha})=tr(u)=±2$, $u^{\alpha}$ is still in $U$. Therefore $U=C_S(u)=C_S(u^{\alpha})=C_S(u)^{\alpha}=U^{\alpha}$ and so $\alpha$ is in $N_S(U)=B$. 
It follows $N_S(K\cap B) \leq B$ and finally $N_K(K\cap B)=K\cap B$.
By the normalizer condition $K\cap B$ cannot be proper in $K$, then $K\leq B$.

It is known (see for example \cite[Lemma 0.1.10] {Wagner}) that if $K$ is a nilpotent group and $H\unlhd K$ a non trivial normal subgroup, then $H\cap Z(K)$ is not trivial. If we assume that $K\nleq U^+$, since $K\leq B=N_S(U^+)$, $K\cap U^+$ is normal in $K$, and so $K\cap U^+$ contains a non trivial element $x$ of the center $Z(K)$. For $q\in K\backslash U^+$, there are $t\in Q_1\backslash \{I\}$ and $u\in U$ such that $q=tu$. We have $[x,q]=I$ so $[x,t]=I$, so $t=-I$ because $C_S(x)=U$. Therefore $K\leq U$.
Since $K$ is maximal nilpotent and $U$ abelian, $K=U$. But $U$ is not a Cartan subgroup, because it is of infinite index in its normalizer $B$. A contradiction.

Since $K$ does not contain a unipotent element, $K$ intersects a conjugate of $Q_1$ or of one of the $Q_{\delta}^{\mu}$ (for $\delta \in \Qp^{\times}\backslash (\Qp^{\times})^2$ and $\mu\in GL_2(\Qp)$) by Corollary \ref{union}, we note $Q$ this subgroup. Let us show that $K=Q$.
Let be $x$ in $K\cap Q$, and $\alpha\in N_K(K\cap Q)$, then $x^{\alpha}\in Q$, and, by lemma \ref{centre}, $Q=C_S(x^{\alpha})=C_S(x)^{\alpha}=Q^{\alpha}$. Thus $\alpha \in N_S(Q)$, and $N_K(K\cap Q)\leq N_S(Q)$.

\begin{description}
 	 \item[1rst case] $Q$ is a conjugate of $Q_1$, then $N_S(Q)=Q\cdot<\omega'>$ where $\omega'=\omega^g$ if $Q=Q_1^g$. We have also $\omega'^2\in Q$ and $t^{\omega'}=t^{-1}$ for $t\in Q$. %First we can easily see that the cartan subgroup $K$ must be infinite. Otherwise the subgroup of $K$ generated by one element of $Q$ will be of infinite index in his normalizer $Q$ in $SL_2(\Qp)$.
One can check that $N_S(Q\cdot<\omega'>)=Q\cdot<\omega'>$, if $\omega' \in K$ then $N_K(Q\cdot<\omega'>\cap K)=Q\cdot<\omega'>\cap K$, by normalizer condition $K\leq Q\cdot<\omega'>$. If we note $n$ the nilpotency classe of K, and $t\in K\cap Q$ then $[\omega', \omega', ... , \omega' , t]=t^{2^n}=1$, so $t $ is an $n^{th}$ root of unity, so $K\cap Q$ and $K=(K\cap Q)\cdot <\omega>$ are finite. 
 A contradiction, so $\omega'\nin K$. Then $N_K(Q\cap K)\leq Q\cap K$, it follows by normalizer condition that $K\leq Q$, and by maximality of $K$, $K=Q$.
	  \item[2nd case] $Q$ is a conjuguate of $Q_{\delta}$ (for $\delta \in \Qp^{\times}\backslash (\Qp^{\times})^2$), then $N_S(Q)=Q$. It follows similarly that $K=Q$.
  \end{description}
\end{proof}

%GENEROSITY 
\subsection*{Generosity of the Cartan subgroups}
Our purpose is now to show the generosity of the Cartan subgroup $Q_1$. It follows from the next more general proposition :

\begin{prop}
\begin{enumerate}
  \item The set $W=\{A\in SL_2(\Qp)\mid v_p(tr(A))<0\}$ is generic in $SL_2(\Qp)$.
  \item The set $W'=\{A\in SL_2(\Qp) \mid v_p(tr(A))\geq 0\}$ is not generic in $SL_2(\Qp)$.\end{enumerate}
\end{prop}

\begin{proof}
1. 
We consider the matrices : 
$$A_1 = I ,\quad A_2= \left(\begin{array}{cc}0 & 1 \\-1 & 0\end{array}\right), \quad A_3=\left(\begin{array}{cc}a^{-1} & 0 \\0 & a\end{array}\right) \quad \mbox{ and } \quad A_4=\left(\begin{array}{cc}0 & -b^{-1} \\b & 0\end{array}\right)$$
with $v_p(a)>0$ and $v_p(b)>0$.

We show that $SL_2(\Qp)=\bigcup_{i=1}^{4} A_iW$.
Suppose there exists  $$M=\left(\begin{array}{cc}x & y \\u & v\end{array}\right)\in SL_2(\Qp)$$ such that $M\nin \bigcup_{i=1}^{4} A_iW$.

Since $M\nin A_1W\bigcup A_2W$, we have $x+v =\varepsilon$ and $y-u=\delta$ with $v_p(\varepsilon)\geq 0$ and $v_p(\delta)\geq 0$. 
Since $M\nin A_3W$, we have $ax+a^{-1}v=\eta$ with $v_p(\eta)\geq 0$.
We deduce $a(\varepsilon-v )+a^{-1}v=\eta$ and $v=\frac{\eta-a\varepsilon}{a^{-1}-a}$.
Similarly, it follows from $M\nin A_4W$ that $u=\frac{\theta-b\delta}{b^{-1}-b}$ with some $\theta$ such that $v_p(\theta)\geq0$. 

Since $v_p(a)>0$, we have $v_p(a+a^{-1})<0$. 
From $v_p(\eta - a\varepsilon)\geq min\{v_p(\eta);v_p(a\varepsilon)\}\geq0$,
we deduce that $v_p(v)=v_p(\frac{\eta-a\varepsilon}{a+a^{-1}})=v_p(\eta-a\varepsilon)-v_p(a+a^{-1})>0$. 
Similarly $v_p(u)>0$. 
It follows that $v_p(x)=v_p(\varepsilon -v)\geq 0$ and $v_p(y)\geq0$.

Therefore $v_p(det(M))=v_p(xv-uy)\geq min\{v_p(xv),v_p(uy)\}>0$ and thus $det(M)\neq 1$, a contradiction .

2.  
We show that the family of matrices $(M_x)_{x\in\Qp^{\times}}$ cannot be covered by finitely many $SL_2(\Qp)$-translates of $W'$, where :
$$M_x =\left(\begin{array}{cc}x & 0 \\0 & x^{-1}\end{array}\right)$$
Let  $A=\left(\begin{array}{cc}a & b \\c & d\end{array}\right) \in SL_2(\Qp)$. 
Then $tr(A^{-1}M_x)=dx+ax^{-1}$. If $v_p(x)>max\{|v_p(a)|, |v_p(d)|\}$ then $v_p(tr(A^{-1}M_x))<0$ and $M_x\nin AW'$. 

Therefore for every finite family $\{A_j\}_{i\leq n}$, there exist $x\in \Qp$ such that $M_x\nin \bigcup_{j=1}^{n} A_jW'$.
\end{proof}

\begin{rmk}
We remark that the sets $W$ and $W'$ form a partition of $SL_2(\Qp)$. They are both definable in the field language because the valuation $v_p$ is definable in $\Qp$.
\end{rmk}

\begin{lem}
$W\subseteq Q_1^{SL_2(\Qp)}$ and for $\delta\in\Qp^{\times}\backslash (\Qp^{\times})^2$ and $\mu\in GL_2(\Qp)$, $Q_{\delta}^{\mu \cdot SL_2(\Qp)}\subseteq W'$
\end{lem}

\begin{proof}
Let be $A\in SL_2(\Qp)$ with $v_p(tr(A))<0$.
\\For $p\neq 2$, since $v_p(tr(A))<0$,  $v_p(tr(A)^2-4)=2v_p(tr(A))$ and $ac(tr(A)^2-4)=ac(tr(A)^2)$, so $tr(A)^2-4$ is a square in $\Qp$. 
\\For $p=2$, we can write $tr(A)=2^nu$ with $n\in \Z$ and $u\in \Zp^{\times}$. Then $tr(A)^2-4=2^{2n}(u^2-4\cdot 2^{-2n})$. Since $n\leq -1$,  $u^2-4\cdot 2^{-2n} \equiv u^2 \equiv 1 (\mbox{mod } 8)$, so $tr(A)^2-A\in (\mathbb{Q}_2^{\times})^2$.

In all cases, by the proposition \ref{tr}, $W\subseteq Q_1^{SL_2(\Qp)}$ and ,by complementarity, $Q_{\delta}^{\mu\cdot SL_2(\Qp)}\subseteq W'$.
\end{proof}

We can now conclude with the following corollary, similar to \cite[Remark 9.8]{cartan} :

\begin{cor}
\begin{enumerate}
  \item The Cartan subgroup $Q_1$ is generous in $SL_2(\Qp)$.
  \item The Cartan subgroups $Q_{\delta}^{\mu}$ (for $\delta\in\Qp^{\times}\backslash (\Qp^{\times})^2$ and $\mu\in GL_2(\Qp)$) are not generous in $SL_2(\Qp)$. 
\end{enumerate}
\end{cor}

\nocite{*}
\bibliographystyle{plain}
\bibliography{truc}

\end{document}